\documentclass{article}
\usepackage{amsmath,amsthm,amsfonts,rotating}
\numberwithin{equation}{section}

\usepackage{amssymb}
\usepackage{rotating}
\usepackage{float}
\usepackage{graphics}
\usepackage{pstricks}
\usepackage{rotating}

\newtheorem{thm}{Theorem}[section]

\newtheorem{cor}[thm]{Corollary}
\newtheorem{lem}[thm]{Lemma}

\theoremstyle{definition}
\newtheorem{exa}{Example}
\newtheorem{defi}{Definition}[section]

\DeclareMathOperator{\Int}{Int}
\newcommand{\R}{\mathbb{R}} 
\newcommand{\N}{\mathbb{N}} 

\newcommand{\tlim}{\displaystyle\lim}

\newcommand{\defrac}{\displaystyle\frac}

\begin{document}
\title{Cone metric spaces and fixed point theorems of $T$-Kannan contractive mappings\thanks{Keywords: Fixed point, $T-$Kannan contractive mapping, complete cone metric space. Subjclass: 46J10, 46J15, 47H10. }}

\author{Jos\'e R. Morales$^a$ and Edixon Rojas$^b$\\\\
\small Department of Mathematics, Faculty of Science\\
\small University of Los Andes, M\'erida-5101, Venezuela.\\ 
\small $^a$moralesj@ula.ve,  $^b$edixonr@ula.ve
}

\date{}

\maketitle

\begin{abstract}
The purpose of this paper is to obtain sufficient conditions for the existence of a unique fixed point of $T$-Kannan type mappings on complete cone metric spaces depended on another function.
\end{abstract}



\section{Introduction}

In the paper \cite{HZ}, Guang and Xian  generalized the notion of  metric spaces, replacing the set of real numbers by an ordered Banach space defining in this way a cone metric space. These authors also described the convergence of sequences in this cone metric spaces and introduce the corresponding notion of completeness. Afterwards, they prove some fixed point theorems of contractive mappings on complete cone metric spaces. Posteriorly, some of the mentioned results were obtained  by Sh. Rezapour and R. Hambarani in \cite{RH} omitting the assumption of normality on the cone.

On the other hand, A. Beiranvand, S. Moradi, M. Omid and H. Pazandeh \cite{BMOP} introduce the classes of $T$-Contraction and $T-$Contractive functions, extending the Banach contraction principle and the Edelstein's fixed point theorem. S. Moradi in \cite{KM} introduce the $T$-Kannan contractive mapping which extend the well known Kannan's fixed point theorem given in \cite{Ka}.

The purpose of this paper is to analyze the existence (and uniqueness) of fixed points of  $T$-Kannan type contractive mappings $S$ defined on a complete cone metric space $(M,d)$, as well as, $T$-Chaterjea mappings which are introduced here (see Definition
\ref{def classes}). Our results generalize the respective theorems given in \cite{HZ} and \cite{KM}.

\section{Definitions and preliminary results}

In this section we recall the definition of cone metric space and some of their properties (see c.f.,\, \cite{HZ}). The following notions will be used in order to prove the main results.

\begin{defi}
Let $E$ be a real Banach space and $P$ a subset of $E$. The set $P$ is called a cone if and only if:
\begin{enumerate}
\item[P1-] $P$ is closed, non empty and $P\neq \{0\};$
\item[P2-] $a,b\in \R,\,\, a,b\geq 0,\,\, x,y\in P\Longrightarrow ax+by\in P;$
\item[P3-] $x\in P$ and $-x\in P\Rightarrow x=0$.
\end{enumerate}
Given a cone $P\subset E,$ we define a partial ordering $\leq$ with respect to $P$ by $x\leq y$ if and only if $y-x\in P.$ We write $x<y$ to indicate that $x\leq y$ but $x\neq y,$ while $x\ll y$ if and only if $y-x\in \Int P$, where $\Int P$ denotes the interior of the set $P$.
\end{defi}

\begin{defi}
Let $E$ be a Banach space and $P\subset E$ a cone. The cone $P$ is called normal if there is a number $K>0$ such that for all $x,y\in E,\,\, 0\leq x\leq y$ implies $\|x\|\leq K\|y\|.$ The least positive number $K$ satisfying the above inequality is called the normal constant of $P.$
\end{defi}
In the following we always suppose that $E$ is a Banach space, $P$ is a cone in $E$ with $\Int P\neq \emptyset$ and $\leq$ is partial ordering with respect to $P.$
\begin{defi}
Let $M$ be a non empty set. Suppose that the mapping $d: M\times M\longrightarrow E$ satisfies:
\begin{enumerate}
\item[d1-] $0<d(x,y)$ for all $x,y\in M$ and $d(x,y)=0$ if and only if $x=y;$
\item[d2-] $d(x,y)=d(y,x)$ for all $x,y\in M;$
\item[d3-] $d(x,y)\leq d(x,z)+d(y,z)$ for all $x,y,z\in M.$
\end{enumerate}
 Then, $d$ is called a cone metric on $M$ and $(M,d)$ is called a cone metric space.
\end{defi}
 Notice that the notion of cone metric space is more general than the corresponding of metric space. Examples of cone metric spaces can be found in \cite{HZ} and \cite{RH}.

\begin{defi}
Let $(M,d)$ be a cone metric space. Let $(x_n)$ be a sequence in $M$ and $x\in M$.
\begin{enumerate}
\item[(i)] $(x_n)$ converges to $x$ if for every $c\in E$ with $0\ll c$, there is an $n_0$ such that
           for all $n\geq n_0,\,\, d(x_n,x)\ll c.$ We denote this by $\tlim_{n\rightarrow \infty} x_n=x$ or $x_n\rightarrow x,\,\, (n\rightarrow \infty)$.

\item[(ii)] If for any $c\in E$ with $0\ll c,$ there is an $n_0$ such that for all $n,m\geq n_0,\,\,
            d(x_n,x_m)\ll c$, then $(x_n)$ is called a Cauchy sequence in $M$.
\end{enumerate}
$(M,d)$ is called a complete cone metric space, if every Cauchy sequence in $M$ is convergent in $M$.
\end{defi}
The following result will be useful for us to prove our main results.

\begin{lem}\textup{\cite{HZ}}
Let $(M,d)$ be a cone metric space, $P\subset E$ a normal cone with normal constant $K.$ Let $(x_n),\, (y_n)$ be sequences in $M$ and $x,y\in M.$
\begin{enumerate}
\item[\textup{(i)}] $(x_n)$ converges to $x$ if and only if $\tlim_{n\rightarrow \infty} d(x_n,x)=0;$
\item[\textup{(ii)}] If $(x_n)$ converges to $x$ and $(x_n)$ converges to $y$ then $x=y.$ That is the limit of
            $(x_n)$ is unique;
\item[\textup{(iii)}] If $(x_n)$ converges to $x$, then $(x_n)$ is Cauchy sequence;
\item[\textup{(iv)}] $(x_n)$ is a Cauchy sequence if and only if
\begin{equation*}
\tlim_{n,m\rightarrow \infty} d(x_n,x_m)=0;
\end{equation*}
\item[\textup{(v)}] If $x_n\longrightarrow x$ and $y_n\longrightarrow y,\,\, (n\rightarrow \infty)$ then
\begin{equation*}
d(x_n,y_n)\longrightarrow d(x,y).
\end{equation*}
\end{enumerate}
\end{lem}

\begin{defi}
Let $(M,d)$ be a cone metric space, $P$ a normal cone with normal constant $K$ and $T: M\longrightarrow M.$ Then
\begin{enumerate}
\item[(i)] $T$ is said to be continuous if $\tlim_{n\rightarrow \infty} x_n=x$ implies that $\tlim_{n
           \rightarrow \infty} T(x_n)=T(x),$ for all $(x_n)$ in $M;$
\item[(ii)] $T$ is said to be subsequentially convergent, if we have for every sequence $(y_n)$ that
            $T(y_n)$ is convergent, implies $(y_n)$ has a convergent subsequence;
\item[(iii)] $T$ is said to be sequentially convergent if we have, for every sequence $(y_n),$ if
             $T(y_n)$ is convergent, then $(y_n)$ also is convergent.
\end{enumerate}
\end{defi}

\section{$T$-Kannan and $T$-Chatterjea contractions: their fixed points on cone metric spaces}

The following theorems are the main results of this paper. First, we are going to introduce some new definitions on cone metric spaces which are based in the ideas of S. Moradi \cite{KM}.
\begin{defi}\label{def classes}
Let $(M,d)$ be a cone metric space and $T,S: M\longrightarrow M$ two functions.
\begin{enumerate}
\item[K1-] A mapping $S$ is said to be a $T-$Kannan contraction, ($TK_1-$Contraction) if there is $b\in [0,1/2)$ constant such
\begin{equation*}
d(TSx,TSy)\leq b[d(Tx,TSx)+d(Ty,TSy)]
\end{equation*}
for all $x,y\in M$.
\item[K2-] A mapping $S$ is said to be a $T$-Chatterjea contraction, ($TK_2-$Contraction) if there is
            $c\in [0,1/2)$ constant such that
\begin{equation*}
d(TSx,TSy)\leq c[d(Tx,TSy)+d(Ty,TSx)]
\end{equation*}
 for all $x,y\in M$.
\end{enumerate}
\end{defi}

\begin{exa}
Let $E=\left(C_{[0,1]},\R\right)$, $P=\{\varphi\in E\,/\, \varphi\geq 0\}\subset E$, $M=\R$ and $d: M\times M\longrightarrow E$ defined by $d(x,y)=|x-y|e^t,$ where $e^t\in E$. Then $(M,d)$ is a cone metric space. We consider the functions $T,S: M\longrightarrow M$ defined by $Tx=x^2$ and $Sx=\defrac{x}{2}.$ Then
\begin{equation*}\begin{array}{ccl}
d(TSx,TSy) &=& |TSx-TSy|e^t=\left|\defrac{x^2}{4}-\defrac{y^2}{4}\right|e^t\\ \\
 &\leq& \defrac{1}{3} \bigg[|Tx-TSx|+|y-TSy|\bigg]e^t\\ \\
 &=& \defrac{1}{3}[d(Tx,TSx)+d(y,TSy)].
\end{array}
\end{equation*}
Therefore, $S$ is a $TK_1-$Contraction. Moreover, it is not difficult to show that $S$ is besides a $TK_2-$Contraction too.
\end{exa}
The following result extend  Theorem 3 of \cite{HZ} and Theorem 2.1 of \cite{KM}.

\begin{thm}
Let $(M,d)$ be a complete cone metric space, $P$ be a normal cone with normal constant $K$, in addition let $T: M\longrightarrow M$ be a one to one, continuous function and $S: M\longrightarrow M$ a $TK_1-$Contraction. Then,
\begin{enumerate}
\item[\textup{(1)}] For every $x_0\in M$
\begin{equation*}
\tlim_{n\rightarrow \infty} d(TS^n x_0, TS^{n+1}x_0)=0;
\end{equation*}

\item[\textup{(2)}] There is $v\in M$ such that
\begin{equation*}
\tlim_{n\rightarrow \infty} TS^n x_0=v;
\end{equation*}

\item[\textup{(3)}] If $T$ is subsequentially convergent, then $(S^nx_0)$ has a convergent subsequence;

\item[\textup{(4)}] There is a unique $u\in M$ such that
\begin{equation*}
Su=u;
\end{equation*}

\item[\textup{(5)}] If $T$ is sequentially convergent, then for each $x_0\in M$ the iterate sequence $(S^nx_0)$
      converges to $u$.
\end{enumerate}
\end{thm}
\begin{proof}
Let $x_0$ be an arbitrary point in $M$. We define the iterate sequence $(x_n)$ by $x_{n+1}=Sx_n=S^nx_0$. We have
\begin{equation*}
\begin{array}{ccl}
d(Tx_n,Tx_{n+1}) &=& d(TSx_{n-1},TSx_n) \\ &\leq& b[d(Tx_{n-1},TSx_{n-1})+d(Tx_n,TSx_n)]
\end{array}
\end{equation*}
so, $d(Tx_n,Tx_{n+1})\leq \defrac{b}{1-b}d(Tx_{n-1},Tx_n)$ and we can conclude, by repeating the same argument, that
\begin{equation}\label{eq3.1}
d(TS^nx_0, TS^{n+1}x_0)\leq \left(\defrac{b}{1-b}\right)^n d(Tx_0,TSx_0).
\end{equation}
From \eqref{eq3.1} we have,
\begin{equation*}
\|d(TS^nx_0, TS^{n+1}x_0)\|\leq \left(\defrac{b}{1-b}\right)^n K\|d(Tx_0,TSx_0)\|
\end{equation*}
where $K$ is the normal constant of $E$. By inequality above we get
\begin{equation*}
\tlim_{n\rightarrow \infty} \|d(TS^nx_0, TS^{n+1}x_0)\|=0
\end{equation*}
hence,
\begin{equation}\label{eq3.2}
\tlim_{n\rightarrow \infty} d(TS^nx_0, TS^{n+1}x_0)=0.
\end{equation}
By inequality \eqref{eq3.1}, for every $m,n\in \N$ with $m>n$ we have,
\begin{equation*}\begin{array}{ccl}
d(Tx_n, Tx_m) &\leq& d(Tx_n, Tx_{n+1})+\ldots+ d(Tx_{m-1}, Tx_m)\\ \\ &\leq& \left[\left(\defrac{b}{1-b}\right)^n+\ldots+\left(\defrac{b}{1-b}\right)^{m-1}\right]d(Tx_0,TSx_0)\\ \\ &=& \left(\defrac{b}{1-b}\right)^n \defrac{1}{1-\left(\defrac{b}{1-b}\right)}d(Tx_0,TSx_0)
\end{array}
\end{equation*}
so,
\begin{equation}\label{eq3.3}
d(TS^n x_0, TS^mx_0)\leq \left(\defrac{b}{1-b}\right)^n
\defrac{1}{1-\left(\defrac{b}{1-b}\right)}d(Tx_0,TSx_0)
\end{equation}
from \eqref{eq3.3} we have,
\begin{equation*}\|d(TS^nx_0, TS^mx_0)\|\leq \left(\defrac{b}{1-b}\right)^n
\defrac{K}{1-\left(\defrac{b}{1-b}\right)}\|d(Tx_0,TSx_0)\|
\end{equation*}
where $K$ is the normal constant of $E$. Taking limit and keeping in mind that $\frac{b}{1-b}<1$, we obtain
\begin{equation*}\tlim_{n,m\rightarrow \infty} \|d(TS^nx_0, TS^mx_0)\|=0,
 \end{equation*}
 in this way we have, $\tlim_{n\rightarrow \infty} d(TS^nx_0, TS^mx_0)=0,$ which implies that $(TS^nx_0)$ is a Cauchy sequence in $M$. Since $M$ is a complete cone metric space, then there is $v\in M$ such that
\begin{equation}\label{eq3.4}
\tlim_{n\rightarrow \infty} TS^nx_0=v.
\end{equation}
Now, if $T$ is subsequentially convergent, $(S^nx_0)$ has a convergent subsequence. So, there are $u\in M$ and $(x_{n_i})$ such that
\begin{equation}\label{eq3.5}
\tlim_{i\rightarrow \infty} S^{n_i}x_0=u.
\end{equation}
Since $T$ is continuous and by \eqref{eq3.5} we obtain
\begin{equation}\label{eq3.6}
\tlim_{i\rightarrow \infty} TS^{n_i}x_0=Tu
\end{equation}
 by \eqref{eq3.4} and \eqref{eq3.6} we conclude that
\begin{equation}\label{eq3.7}
Tu=v.
\end{equation}
On the other hand,
\begin{equation*}\begin{array}{ccl}
d(TSu,Tu) &\leq& d(TSu,TS^{n_i}(x_0))+d(TS^{n_i}x_0, TS^{n_i+1}x_0)+d(TS^{n_i+1}x_0,Tu)\\
&\leq& b[d(Tu,TSu)+d(TS^{n_i-1}x_0, TS^{n_i}x_0)]\\
&&+\left(\defrac{b}{1-b}\right)^{n_i}d(Tx_0,TSx_0)+d(TS^{n_i+1}x_0,Tu)
\end{array}
\end{equation*}
hence,
\begin{equation*}\begin{array}{ccl}
d(TSu,Tu) &\leq& \defrac{b}{1-b}d(TS^{n_i-1}x_0, TS^{n_i}x_0)+\defrac{1}{1-b}\left(\defrac{b}{1-b}\right)^{n_i}d(TSx_0, Tx_0)\\
&&+ \defrac{1}{1-b}d(TS^{n_i+1}x_0,Tu)
\end{array}
\end{equation*}
thus,
\begin{equation*}\begin{array}{ccl}
\|d(TSu,Tu)\| &\leq& \defrac{bK}{1-b}\|d(TS^{n_i-1}x_0, TS^{n_i}x_0)\|\\
&&+\defrac{K}{1-b}\left(\defrac{b}{1-b}\right)^{n_i}\|(TSx_0,Tx_0)\|\\
&&+ \defrac{K}{1-b}\|d(TS^{n_i+1}x_0,Tu)\|\longrightarrow 0\quad (i\rightarrow \infty)
\end{array}
\end{equation*}
 where $K$ is the normal constant of $M$. The convergence above give us that $d(TSu,Tu)=0$, which implies the equality $TSu=Tu$. Since $T$ is one to one, then $Su=u$, consequently $S$ has a fixed point. Because $S$ is a $TK_1-$Contraction we have
 \begin{equation*}
 d(TSu,TSv)\leq b[d(Tu,TSu)+d(Tv,TSv)].
 \end{equation*}
  If $v$ is another fixed point of $S$, then from the injectivity of $T$ we get $Su=Sv$, or which is the same, the fixed point is unique.  Finally, if $T$ is sequentially convergent, by replacing $(n)$ for $(n_i)$ we conclude that
\begin{equation*}\tlim_{n\rightarrow \infty} S^nx_0=u.
\end{equation*}
This shows that $(S^nx_0)$ converges to the fixed point of $S$.
\end{proof}
If we take $Tx=x$ in Theorem 3.3 we get the following,
\begin{cor}[\cite{HZ}, Theorem 3]
Let $(M,d)$ be a complete cone metric space, $P$ a normal cone with normal constant $K$. Suppose that the mapping $S: M\longrightarrow M$ satisfies the contractive condition
\begin{equation*}
d(Sx,Sy)\leq b[d(x,Sx)+d(y,Sy)]
\end{equation*}
for all $x,y\in M$ and $b\in[0,\frac{1}{2})$. Then $S$ has a unique fixed point in $M$ and for any $x_0\in M,$ iterative sequence $(S^n x_0)$ converges to the fixed point.
\end{cor}
If we take $M=\R$ in the Theorem 3.3 we obtain the following:

\begin{cor}[\cite{KM}, Theorem 2.1]
Let $(M,d)$ be a complete metric space and $T,S: M\longrightarrow M$ be mappings such that $T$ is continuous, one to one and subsequentially convergent. If $b\in [0,1/2)$ and
\begin{equation*}
d(TSx,TSy)\leq b[d(Tx,TSx)+d(Ty,TSy)]
 \end{equation*}
for all $x,y\in M$. Then $S$ has a unique point and if $T$ is sequentially convergent, then for every $x_0\in M$ the sequence iterates $(S^nx_0)$ converges to the fixed point of $S$.
\end{cor}
If we take $M=\R$ and $Tx=x$ in the Theorem 3.3, we obtain the following result given by Kannan in \cite{Ka}.

\begin{cor}
Let $(M,d)$ be a complete metric space and $S: M\longrightarrow M$ a mapping. If there exists $b\in [0,1/2)$ such that
\begin{equation*}
d(Sx,Sy)\leq b[d(x,Sx)+d(y,Sy)]
 \end{equation*}
for all $x,y\in M$. Then $S$ has a unique fixed point and $(S^nx_0)$ converges to the fixed point of $S$ for all $x_0\in M.$
\end{cor}
The following result extend the Theorem 4 of \cite{HZ}.
\begin{thm}
Let $(M,d)$ be a complete cone metric space, $P$ be a normal cone with normal constant $K$, let in addition $T: M\longrightarrow M$ be a continuous, one to one function and $S: M\longrightarrow M$ a $TK_2-$ Contraction. Then
\begin{enumerate}
\item[\textup{(1)}] For every $x_0\in M$
\begin{equation*}
\tlim_{n\rightarrow \infty} d(TS^n x_0, TS^{n+1}x_0)=0.
\end{equation*}
\item[\textup{(2)}] There is $v\in M$ such that
\begin{equation*}
\tlim_{n\rightarrow \infty} TS^nx_0=v.
\end{equation*}
\item[\textup{(3)}] If $T$ is subsequentially convergent, then $(S^nx_0)$ has a convergent subsequence.

\item[\textup{(4)}] There is a unique $u\in M$ such that
\begin{equation*}
Su=u.
\end{equation*}

\item[\textup{(5)}] If $T$ is sequentially convergent, then for each $x_0\in M$, $(S^n(x_0))$ converges to $u$.
\end{enumerate}
\end{thm}

\begin{proof} Let $x_0$ be an arbitrary point in $M$. We define the iterative sequence $(x_n)$ by $x_{n+1}=Sx_n=S^nx_0$. Since $S$ is a $TK_2-$Contraction, we have
\begin{eqnarray*}
d(TSx_n,TSx_{n+1})&\leq&c[d(Tx_n,TSx_{n+1})+d(Tx_{n+1},TSx_n)]\\
                  &\leq&c[d(TSx_{n-1},TSx_n)+d(TSx_n,TSx_{n+1})].
\end{eqnarray*}
Thus,
\begin{equation*}
d(TSx_n,TSx_{n+1})\leq\dfrac{c}{1-c}d(TSx_{n-1},TSx_n)=hd(TSx_{n-1},TSx_n)
\end{equation*}
where $h:=\frac{c}{1-c}$. Recursively, we obtain
\begin{equation}\label{eq3.8}
d(TSx_n, TSx_{n+1})\leq h^nd(TSx_0, TSx_1)
\end{equation}
therefore,
\begin{equation*}
\|d(TSx_n, TSx_{n+1})\|\leq h^nK\|d(TSx_0, TSx_1)\|
\end{equation*}
where $K$ is the normal constant of $M$. Hence
\begin{equation*}
\tlim_{n\rightarrow \infty} \|d(TSx_n, TSx_{n+1})\|=0,
\end{equation*}
this implies that
\begin{equation*}
\tlim_{n\rightarrow \infty} d(TS^n x_0, TS^{n+1}x_0)=0.
\end{equation*}
By \eqref{eq3.8}, for every $m,n\in \N$ with $n>m$ we have,
\begin{eqnarray*}
d(TSx_m,TSx_n)&\leq&d(TSx_n,TSx_{n+1})+\dots+d(TSx_{m-1},TSx_m)\\
               &\leq&[h^{n-1}+h^{n-2}+\dots+h^m]d(TSx_0,TSx_1)\\
               &\leq&\dfrac{h^m}{1-h}d(TSx_0,TSx_1),
\end{eqnarray*}
taking norm we get
\begin{equation*}
\|d(TSx_m,TSx_n)\|\leq \dfrac{h^m}{1-h}K\|d(TSx_0,TSx_1)\|,
\end{equation*}
consequently, we have
\begin{equation*}
\tlim_{n,m\rightarrow \infty} d(TSx_n, TSx_m)=0
 \end{equation*}
 hence $(TS^nx_0)$ is a Cauchy sequence in $M$ and since $M$ is a complete cone metric space, there is $v\in M$ such that
\begin{equation*}
\tlim_{n\rightarrow \infty} TS^nx_0=v.
\end{equation*}
 The rest of the proof is similar to the proof of Theorem 3.3.
\end{proof}

Notice that if we take $Tx=x$ in the previous theorem, we obtain the Theorem 4 of \cite{HZ}.
\begin{exa}
Let $E=(C_{[0,1]},\R)$, $P=\{\varphi\in E\,/\,\varphi\geq 0\}\subset E,\,\, M=[0,1]$ and $d(x,y)=|x-y|e^t,\,\, e^t\in E.$ It is clear that $(M,d)$ is a complete cone metric space.

Let $T,S: M\longrightarrow M$ be two functions defined by $Tx=x^2$ and $Sx=\defrac{x}{2}.$ Then,
\begin{enumerate}
\item[(i)] $T$ and $S$ are continuous mappings.

\item[(ii)] It is clear that $S$ is a contraction function

\item[(iii)] $S$ is not a Kannan contraction, that is, $S$ is not $K_1-$Contraction.

\item[(iv)] $S$ is a $TK_1-$Contraction, because
\begin{equation*}\begin{array}{ccl}
d(TSx, TS y) &=& |TSx-TSy|e^t=\left|\defrac{x^2}{4}-\defrac{y^2}{4}\right|e^t\\
 &\leq& \defrac{1}{3}\bigg(|T x-TSx)|+|Ty-TSy|\bigg)e^t\\
  &=& \defrac{1}{3}\bigg[d(Tx, TSx)+d(Ty, TSy)\bigg].
\end{array}\end{equation*}
 Therefore, by Theorem 3.3, $v=0$ is the unique fixed point of $S$ in $M$.
\end{enumerate}
\end{exa}

\end{document}